\documentclass{amsart}

\usepackage{fullpage, amsmath, amsthm, amssymb, graphicx, color}
\usepackage{ulem}

\theoremstyle{plain}% default
\newtheorem{theorem}{Theorem}[section]
\newtheorem{lemma}[theorem]{Lemma}

\newtheorem{corollary}[theorem]{Corollary}

\newcommand{\Sum}[3]{\displaystyle\sum_{#1=#2}^{#3}}
\newcommand{\Prod}[3]{\displaystyle\prod_{#1=#2}^{#3}}

\newtheorem{definition}[theorem]{Definition} %[section]
\theoremstyle{remark}
\newtheorem*{remark}{Remark}

\newcommand{\leg}[2]{{\left(\frac{#1}{#2}\right)}}

\author{Matthew S. Mizuhara}
\author{James A. Sellers}
\author{Holly Swisher}

\title{A periodic approach to plane partition congruences }

\thanks{The first and third authors were supported in part by the NSF grant DMS-0852030. }


\begin{document}

\maketitle

%%Rewrite Abstract%%
\begin{abstract}
Ramanujan's celebrated congruences of the partition function $p(n)$ have inspired a vast amount of results on various partition functions.  Kwong's work on periodicity of rational polynomial functions yields a general theorem used to establish congruences for restricted plane partitions.  This theorem provides a novel proof of several classical congruences and establishes two new congruences. We additionally prove several new congruences which do not fit the scope of the theorem, using only elementary techniques, or a relationship to existing multipartition congruences.
\end{abstract}

\section{Introduction}

The subject of partitions has a long and fascinating history, including connections to many areas of mathematics and mathematical physics.  For example, \cite{A-O-Notices} and \cite{And98} provide a nice overview of the history and theory of partitions.  The generalization of partitions to $k$-component plane partitions has been a rich subject in its own right (see \cite{Stanley} for a nice survey).  We first review partitions and plane partitions, introducing all necessary definitions.
 
\subsection{Partitions and Plane Partitions}\label{planepartitions}
A {\it{partition}} of a positive integer $n$ is defined to be a nonincreasing sequence of positive integers, called {\it{parts}}, that sum to $n$ (often written as a sum).  For $n=0$ we consider the empty set to be the unique ``empty partition" of $0$.  We write  $|\lambda|=n$ to denote that $\lambda$ is a partition of $n$.  For example, the following gives all the partitions $\lambda$, such that $|\lambda|=5$:
\[
5 = 4+1 = 3+2 = 3+1+1 = 2+2+1 = 2+1+1+1= 1+1+1+1+1.
\]
The {\it{partition function}} $p(n)$ counts the total number of partitions of $n$.  In order to define $p(n)$ on all integers we define $p(n)=0$ for $n<0$.  We see from our example above that $p(5)=7$.

To each partition $\lambda$ of $n$, with parts $\lambda_1\geq \ldots \geq \lambda_r$, we associate a {\it Ferrers diagram} by constructing a left-justified array of $n$ cells, where the $i$th row from the top contains $\lambda_i$ cells corresponding to the $i$th part of $\lambda$.  For example, the following is the Ferrers diagram for the partition $2+2+1$ of $5$.
\[
\includegraphics[scale=0.5]{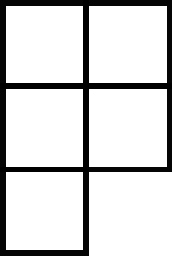}
\]

The generating function for $p(n)$ has the following infinite product form, due to Euler:
\[
\sum_{n=0}^\infty p(n)q^n=\prod_{n=1}^\infty\frac{1}{(1-q^n)}.
\]

One can also consider partitions where the parts are restricted to a specific set $S$ of integers.  These are called {\it restricted partitions}.  Given a set of integers $S$, define the restricted partition function $p(n;S)$ to be the number of partitions of $n$ with parts from $S$.  For example, let $S$ be the set of odd integers.  Then $p(5;S)=3$ since the partitions of $5$ into only odd parts are
\[
5 = 3+1+1 = 1+1+1+1+1.
\]
Generalizing slightly, we allow $S$ to be a multiset where repetition of an integer corresponds to a distinct coloring of the integer (where coloring also equips an ordering among like integers).  For example, let $S=\{1,{\bf 2},2, 5\}$ where ${\bf 2}>2$. Then $p(5;S)=7$ since
\[
5 ={\bf 2}+ {\bf 2}+1={\bf 2}+2+1 = 2+2+1= {\bf 2}+1+1+1 = 2+1+1+1  = 1+1+1+1+1
\]
are all partitions formed from $S$.

Partitions can  be generalized in a number of natural ways. A {\it{plane partition}} $\lambda$ of a positive integer $n$ is a two-dimensional array of positive integers $n_{i,j}$ that sum to $n$, such that the array is the Ferrers diagram of a partition, and the entries are non-increasing from left to right and also from top to bottom.  Letting $i$ denote the row and $j$ the column of $n_{i,j}$, this means that for all $i,j\geq 0$, 
\[
n_{i,j}\geq n_{i+1,j} \mbox{ and } n_{i,j} \geq n_{i, j+1}.
\]
Correspondingly, the entries $n_{i,j}$ are called the {\it parts} of $\lambda$.   

If a plane partition $\lambda$ has largest part $\leq k$, we say it is a {\it{$k$-component plane partition}}. We can visualize $k$-component plane partitions by letting each part represent a height so that $\lambda$ is a $k$-tuple of partitions ``stacked'' one atop the other with heights monotonically increasing up to $n_{1,1}$ in the top-left corner. In this sense, we can think of a $k$-component plane partition as $k$-tuple of partitions (obeying certain rules), where we note that the last partitions in the tuple may be empty sets. As an example, the following is a $5$-component plane partition of $28$ corresponding to the $5$-tuple of partitions $(4+3+1+1, 4+3+1, 4+1, 4, 2)$.

\[
\includegraphics[scale=0.5]{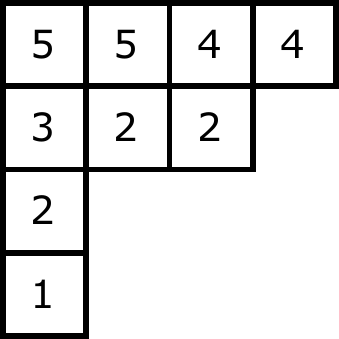}
\]

Let $pl_k(n)$ count the number of $k$-component plane partitions of $n$.  As above, we define $pl_k(n)=0$ for $n<0$, and consider the empty set as the unique $k$-component plane partition of $0$.  Then, the generating function for $pl_k(n)$ has the following infinite product form \cite{And98}:

\begin{equation}\label{PL}
PL_k(q)=\sum_{n=0}^{\infty}pl_k(n)q^n=\prod_{n=1}^{\infty}\frac{1}{(1-q^n)^{\rm{min}(k,\,n)}} = F_k(q) \prod_{n=k}^\infty \frac{1}{(1-q^n)^k},
\end{equation}
where for any positive integer $k\geq 2$, we define $F_k(q)$ by
\[
F_k(q):=\Prod{n}{1}{k-1}\frac{1}{(1-q^n)^n},
\]
and $F_1(q):=1$.
\subsection{Plane Partition Congruences}\label{pcongruences}

Tremendous amounts of work have been done on the study of Ramanujan-type congruences for partitions and generalized partitions of various kinds, using combinatorial methods, q-series analysis, and modular forms.  For example, see \cite{A-B}, \cite{A-O}, \cite{AND08}, \cite{Atkin}, \cite{Boylan}, \cite{GAN67}, \cite{KIM92}, \cite{LSY}, \cite{LMRS},   \cite{Newman}, \cite{Ono} to name a few. 

In this work, using results of Kwong \cite{KWO89} on periodicity of rational functions, we consider the class of plane partition congruences of the form
\begin{equation}\label{congform}
\sum_{i=1}^{s}pl_{k}(\ell n+a_i)\equiv \sum_{j=1}^{t}pl_{k}(\ell n+b_j) \pmod {\ell}, \;\;\; \mbox{for all $n\geq 0$.}
\end{equation}

\begin{definition}
Let $A(q) = \sum_{n\geq 0}\alpha_n q^n \in \mathbb{Z}[[ q ]]$ be a formal power series with integer coefficients, and let $d,\ell$ be positive integers. We say $A(q)$ is {\it periodic with period $d$ modulo $\ell$} if, for all  $n\geq 0$,
\[
\alpha_{n+d}\equiv \alpha_{n} \pmod \ell.
\]
The smallest such period, denoted $\pi_m(A)$, is called the minimal period of $A(q)$ modulo $m$. 
\end{definition}

We prove the following theorem.  

%Specifically, we study the sequence of $pl_{\ell}(n)$, where $\ell$ is a prime. We define $\pi(k)$ be the length of the minimum period of the coefficients of $\frac{1}{(1-q)(1-q^2)^2\dots(1-q^{p-1})^{p-1}}$ modulo $p$. It will be shown that this is well-defined. The purpose of this paper is to prove the following:

\begin{theorem}\label{congtheorem} 
Fix positive integers $s,t$ and nonnegative integers $a_i,\,b_j$ for each $1\leq i\leq s, 1\leq j \leq t$.  For a prime $\ell$, if 
\begin{equation}\label{congeqn}
\sum_{i=1}^{s}pl_{\ell}(\ell n+a_i)\equiv \sum_{j=1}^{t}pl_{\ell}(\ell n+b_j) \pmod {\ell}
\end{equation}
holds for all $n< \frac{\pi_{\ell}(F_\ell)}{\ell}$, then it holds for all $n\geq 0$.
\end{theorem}

Congruences of this type were studied by Gandhi in \cite{GAN67}, and proved using relationships between $pl_k(n)$ and unrestricted multipartition functions, for which many results were known. Although we prove two plane partition congruences by relating them to multipartition congruences in Section \ref{mod5section}, the approach in Theorem \ref{congtheorem} does not rely on previous results of multipartition congruences, but rather on periodicity of plane partition generating functions modulo prime numbers.  The utility of Theorem \ref{congtheorem} is that it provides an upper bound on the number of calculations required to confirm these types of congruences.

In the next section, we introduce necessary preliminaries and discuss the periodicity properties of $PL_k(q)$.  In section \ref{proof} we prove Theorem \ref{congtheorem} and list all confirmed congruences arising from its application. We conclude by proving several other plane partition congruences which do not fall under the scope of Theorem \ref{congtheorem}, some using elementary techniques.

\section{Preliminaries}

In order to state a theorem of Kwong \cite{KWOPart}, we need the following definitions.

\begin{definition}
For an integer $n$ and prime $\ell$, define $\rm{ord}_\ell(n)$ to be the unique nonnegative integer such that $\ell^{\rm{ord}_\ell(n)}\cdot m =n$, where $m$ is an integer and $\ell\nmid m$.  In addition, we call $m$ the $\ell$-free part of $n$.   
\end{definition}

\begin{definition}
Fix a prime $\ell$.  For a finite multiset of positive integers $S$, we define $m_\ell(S)$ to be the $\ell$-free part of lcm$\{n \mid n\in S\}$, and $b_\ell(S)$ to be the least nonnegative integer such that
\[
\ell^{b_\ell(S)}\geq \sum_{n\in S}\ell^{\rm{ord}_\ell(n)}.
\]
When the context is clear we will write $m(S) = m_\ell(S)$, and $b(S) = b_\ell(S)$ for simplicity.
\end{definition}

\begin{theorem}[Kwong \cite{KWOPart}]\label{kwong}
Fix a prime $\ell$, and a finite multiset $S$ of positive integers.  Then, for any positive integer $N$, $A(q)=\sum_{n\geq 0}p(n;S)q^n$ is periodic modulo ${\ell^N}$, with minimal period $\pi_{\ell^N}(A)=\ell^{N+b(S)-1}m(S)$.
\end{theorem}

We can use Theorem \ref{kwong} to decompose the generating function $PL_k(q)$ for $k$-component plane partitions. Recall from \eqref{PL} that 
\[
PL_k(q)=F_k(q)\cdot \Prod{n}{k}{\infty}\frac{1}{(1-q^n)^{k}},
\]
where $F_1(q)=1$, and for $k\geq 2$,
\[
F_k(q):=\Prod{n}{1}{k-1}\frac{1}{(1-q^n)^n} = [1 + q^1 + q^{1+1} + \cdots][1 + q^2 + q^{2+2} + \cdots]^2 \cdots [1 + q^{k-1} + \cdots]^{k-1}.
\]
We observe that $F_k(q)$ generates partitions into parts $1 \leq i \leq k-1$, where each part $i$ can be one of $i$ colors.  Define the multiset $S_{k}$ to contain $i$ colorings of each positive integer $i<k$, i.e., 
\[
S_{k}:=\{i_j|1\leq i \leq k-1,\, 1\leq j \leq i\}.
\]
Then,
\begin{equation}\label{coeffsF_k}
\sum_{n\geq0}p(n;\,S_{k})q^n=F_{k}(q).
\end{equation}
This gives the following  immediate corollary of Theorem \ref{kwong}.
\begin{corollary}\label{corollary}
Let $\ell$ be prime, and $k,N$ positive integers.  Then the series $F_{k}(q)$ is periodic modulo $\ell^N$ with minimal period $\pi_{\ell^N}(F_{k})=\ell^{N+b(S_{k}) -1}\cdot m(S_{k})$.
\end{corollary}

As a consequence of Corollary \ref{corollary},  \eqref{coeffsF_k} gives that for all $n\geq 0$,
\[
p(n+ \pi_{\ell^N}(F_{k});S_k) \equiv p(n;S_k) \pmod{\ell^N}, 
\]
so for any $n,m\geq 0$ with $n\equiv m \pmod{\pi_{\ell^N}(F_{k})}$, we have
\begin{equation}\label{periodicity}
p(n;S_k) \equiv p(m;S_k) \pmod{\ell^N}.
\end{equation}

%Thus for any prime $\ell$, and positive integers $k,N$, 
%\begin{equation}\label{F}
%F_{k}(q)\equiv \Sum{i}{0}{\infty}\Sum{n}{0}{\pi_{\ell^N}(F_{k})-1} p(n;S_k) q^{\pi_{\ell^N}(F_{k})i+n} \pmod {\ell^N}.
%\end{equation}

Taking $k=\ell$, we can easily calculate the minimum period $\pi_{\ell^{N}}(F_\ell)$. Notice that by definition of $S_\ell$, $\rm{ord}_{\ell}(n)=0$ for any $n\in S_\ell$.  Thus,
\[
\sum_{n\in S_\ell}\ell^{\rm{ord}_{\ell}(n)}= | S_\ell | = \frac{\ell(\ell-1)}{2},
\] 
and $b(S_\ell)$ is the least nonnegative integer such that $\ell^2-\ell \leq 2\ell^{b(S_{\ell})}$.  We see  then that for any $\ell$, $0\leq b(S_\ell)\leq 2$.  We immediately compute that $b(S_2)=0$, $b(S_3)=1$, and $b(S_\ell)=2$ for $\ell\geq 5$.  Also, by construction of $S_\ell$, the $\ell$-free part of $\rm{lcm}\{n\mid n\in S_\ell\}$ is $m(S_\ell)=\rm{lcm}\{1,2,\ldots, \ell -1\}$.  Thus for any positive integer $N$,
\begin{equation}\label{periods}
\pi_{\ell^N}(F_\ell) = \begin{cases}
     2^{N-1} & \text{ if } \ell=2, \\
     3^{N}\cdot2 & \text{ if } \ell=3, \\
     \ell^{N+1}\cdot \rm{lcm}\{1,2,\ldots, \ell -1\} & \text{ if } \ell \geq 5.
\end{cases}
\end{equation}

To conclude this section, we recall the following standard lemma which follows from the binomial series. 
\begin{lemma}\label{equivlemma}
If $\ell$ is prime, then for any positive integer $j$,
\[
(1-q^j)^\ell\equiv (1-q^{j\ell}) \pmod {\ell}.
\]
\end{lemma}

\section{Proof of Theorem \ref{congtheorem}}\label{proof}

Fix a prime $\ell$.  For convenience we will write $\alpha_i := p(i; S_\ell)$, the coefficient of $q^i$ in $F_\ell(q)$.  From (\ref{PL}), \eqref{coeffsF_k}, and Lemma \ref{equivlemma}, we see that

\begin{equation}\label{mkmod}
PL_\ell(q)=\Sum{n}{0}{\infty}pl_{\ell}(n)q^n\equiv \left(\Sum{i}{0}{\infty} \alpha_i q^{i}\right)\left(\prod_{j=\ell}^{\infty}\frac{1}{(1-q^{j\ell })}\right) \pmod {\ell}.
\end{equation}
 The product $\prod_{j=\ell}^{\infty}\frac{1}{(1-q^{\ell j})}$ yields a series in $q$ with exponents which are all multiples of $\ell$. Thus we write 
\begin{equation}\label{betaform}
\prod_{j=\ell}^{\infty}\frac{1}{(1-q^{j \ell})}=\sum_{m\geq 0}\beta_m q^{m\ell}.
\end{equation} 
Notice that $\beta_m\in \mathbb{N}$, and $\beta_0=1$. 
For any $n\geq 0$, and $0\leq k <\ell$, \eqref{mkmod} and \eqref{betaform} give that
\begin{equation}\label{star}
pl_{\ell}(n\ell + k) \equiv \sum_{i=0}^n \alpha_{i\ell + k}\cdot \beta_{n-i}  \pmod{\ell}.
\end{equation}
%\begin{equation}\label{star}
%pl_{\ell}(n\ell + k) \equiv \sum_{i=0}^n \alpha_{i\ell + k}\cdot \beta_{n-i} \equiv \alpha_{n\ell + k} + \sum_{i=0}^{n-1} \alpha_{i\ell + k}\cdot \beta_{n-i}  \pmod{\ell}.
%\end{equation}

%\textcolor{red}{(We have $\alpha_0,\dots,\alpha_{\pi_\ell(F_\ell)}$ defined, but $\alpha_{i\ell+k}$ is not defined here if $i\ell+k$ is sufficiently large.)}

Notice that by (\ref{star}), for any $n\geq 0$, the congruence
\[
\sum_{i=1}^{s}pl_{\ell}(\ell n+a_i)\equiv \sum_{j=1}^{t}pl_{\ell}(\ell n+b_j) \pmod {\ell}
\]
is equivalent to the congruence
\[
\sum_{i=1}^{s} \sum_{r=0}^n \alpha_{r\ell+a_i} \beta_{n-r} \equiv \sum_{j=1}^{t} \sum_{r=0}^n \alpha_{r\ell +b_j}\beta_{n-r} \pmod{\ell},
\]
or in particular
\[
\sum_{r=0}^n \beta_{n-r} \left( \sum_{i=1}^{s} \alpha_{r\ell+a_i} \right)  \equiv  \sum_{r=0}^n \beta_{n-r} \left( \sum_{j=1}^{t}\alpha_{r\ell +b_j}\right) \pmod{\ell}.
\]
Thus to prove any congruence of the form in (\ref{congform}), it suffices to prove that for all $n\geq 0$, 
\[
\sum_{i=1}^{s} \alpha_{n\ell+a_i} \equiv \sum_{j=1}^{t}\alpha_{n\ell +b_j} \pmod{\ell}.
\]  

By the hypotheses of Theorem \ref{congtheorem}, we may assume that   for all $0\leq n< \pi_\ell(F_\ell)/\ell$, 
\[
\sum_{i=1}^{s}pl_{\ell}(\ell n+a_i)\equiv \sum_{j=1}^{t}pl_{\ell}(\ell n+b_j) \pmod {\ell},
\]
where $s,t$ are positive integers, and $0\leq a_i,\,b_j <\ell$.  Thus for all $0\leq n< \pi_\ell(F_\ell)/\ell$,
\begin{equation}\label{star'}
\sum_{r=0}^n \beta_{n-r} \left( \sum_{i=1}^{s} \alpha_{r\ell+a_i} \right)  \equiv  \sum_{r=0}^n \beta_{n-r} \left( \sum_{j=1}^{t}\alpha_{r\ell +b_j}\right) \pmod{\ell}.
\end{equation}
  Letting $n=0$, \eqref{star'} implies that $\sum_{i=1}^{s} \alpha_{a_i} \equiv \sum_{j=1}^{t}\alpha_{b_j} \pmod{\ell}$, and when $n\geq 1$,   
\begin{equation}\label{recursive}
\sum_{i=1}^{s} \alpha_{n\ell+a_i} + \sum_{r=0}^{n-1} \beta_{n-r} \left( \sum_{i=1}^{s} \alpha_{r\ell+a_i} \right)  \equiv \sum_{j=1}^{t}\alpha_{n\ell +b_j} +  \sum_{r=0}^{n-1} \beta_{n-r} \left( \sum_{j=1}^{t}\alpha_{r\ell +b_j}\right) \pmod{\ell}.
\end{equation}
We see recursively from (\ref{recursive}) that for all $0\leq n< \pi_\ell(F_\ell)/\ell$,
\begin{equation}\label{finitecong}
\sum_{i=1}^{s} \alpha_{n\ell+a_i} \equiv \sum_{j=1}^{t}\alpha_{n\ell +b_j} \pmod{\ell}.
\end{equation}
To finish the proof of Theorem \ref{congtheorem} it thus suffices to prove that \eqref{finitecong} holds for all $n\geq 0$.

%To finish the proof of Theorem \ref{congtheorem} it thus suffices to prove the following lemma. 

%\begin{lemma}
%Let $F_\ell(q) = \sum_{n\geq 0} \alpha_n q^n$ as above be periodic modulo a prime $\ell$ with minimal period $\pi_\ell(F_\ell)$.  If  the congruence  
%\[
%\sum_{i=1}^{s} \alpha_{n\ell+a_i} \equiv \sum_{j=1}^{t}\alpha_{n\ell +b_j} \pmod{\ell},
%\]  
%holds for all $0\leq n< \pi_\ell(F_\ell)/\ell$, then  the congruence   holds for all $n\geq 0$.
%\end{lemma}

%\begin{proof}
By  (\ref{periods}), we know that $\pi_2(F_2)=1$, and thus the coefficients of $F_2$ are all congruent modulo $2$.  By  (\ref{periods}), we see that $\pi_\ell(F_\ell)$ is a multiple of $\ell$ whenever $\ell >2$.  Thus, we write $\pi_\ell(F_\ell) = K\ell$ in this case,  for $K\in\mathbb{N}$.  Fix an arbitrary positive integer $n\geq \pi_\ell(F_\ell)/\ell$.  By the division algorithm, we can write $n = xK+y$ for $0\leq y<K$.  Thus,  for each $1\leq i \leq s$, and $1\leq j\leq t$, we have 
\begin{align*}
n\ell + a_i &= x\cdot \pi_\ell(F_\ell) + (y\ell + a_i)\\
n\ell + b_j &= x\cdot \pi_\ell(F_\ell) + (y\ell + b_j),
\end{align*}
 where $0\leq y<K$.     
From this we see that $n\ell + a_i \equiv y\ell + a_i \pmod{\pi_\ell(F_\ell)}$, and $n\ell + b_j \equiv y\ell + b_j \pmod{\pi_\ell(F_\ell)}$, and so by  \eqref{periodicity}, we must have that for each $1\leq i \leq s$, and $1\leq j\leq t$,
\begin{eqnarray*}
\alpha_{n\ell + a_i} &\equiv& \alpha_{y\ell + a_i} \pmod{\ell} \\ \alpha_{n\ell + b_j} &\equiv& \alpha_{y\ell + b_j} \pmod{\ell} .
\end{eqnarray*}
But $y\ell + a_i, y\ell + b_j < (K-1)\ell + \ell =\pi_\ell(F_\ell)$, which implies that $y<\pi_\ell(F_\ell)/\ell$.  Thus we have 
\[
\sum_{i=1}^{s} \alpha_{n\ell+a_i} \equiv \sum_{i=1}^{s} \alpha_{y\ell+a_i} \equiv \sum_{j=1}^{t}\alpha_{y\ell +b_j} \equiv \sum_{j=1}^{t}\alpha_{n\ell +b_j}\pmod{\ell},
\]
as desired.
%\end{proof}

\subsection{Known Congruences}

The following lists all congruences we have established using this theorem.

\begin{theorem}
The following hold for all $n\geq 0$. 
\begin{equation}\label{mod2case}
pl_2(2n+1)\equiv pl_2(2n) \pmod 2
\end{equation}

\begin{equation}\label{mod3case1}
pl_3(3n+2)\equiv 0 \pmod 3
\end{equation}
\begin{equation}\label{mod3case2}
pl_3(3n+1)\equiv pl_3(3n) \pmod 3
\end{equation}

\begin{equation}
pl_5(5n+2)\equiv pl_5(5n+4) \pmod 5
\end{equation}
\begin{equation}\label{mod5case2}
pl_5(5n+1)\equiv pl_5(5n+3) \pmod 5
\end{equation}

\begin{equation}\label{mod7case}
pl_7(7n+2)+pl_7(7n+3)\equiv pl_7(7n+4)+pl_7(7n+5) \pmod 7
\end{equation}
\end{theorem}
\begin{remark}
The equivalences \eqref{mod2case}, \eqref{mod3case2}-\eqref{mod5case2} were known to Gandhi \cite{GAN67} but \eqref{mod3case1} and \eqref{mod7case} are previously unreported in the literature.
\end{remark}

\begin{remark}
To demonstrate an application of Theorem \ref{congtheorem}, we establish \eqref{mod3case1} explicitly. By \eqref{periods} it is sufficient to confirm the congruence for $n< \frac{\pi_3(F_3)}{3} = 2$. Indeed, immediate calculations yield
\begin{align*}
pl_3(2) &= 3\equiv 0 \pmod{3}, \\
pl_3(5) &=21 \equiv 0 \pmod{3}.
\end{align*} 
\end{remark}

\subsection{Congruences involving prime powers}

Note that Lemma \ref{equivlemma} allows us to write expansions of series with support on multiples of primes but that this representation fails for non-primes.  As such Theorem \ref{congtheorem} cannot readily be extended to congruences involving non-primes using the techniques presented in this work. However, using \eqref{mod2case} and some elementary techniques we are able to prove several congruences modulo prime powers.
\begin{theorem}\label{mod4equiv}
The following hold for all $n\geq 0$.
\[
pl_4(4n+1)\equiv pl_4(4n+2)+pl_4(4n+3) \pmod 4
\]
\[
pl_4(4n+3)\equiv 0 \pmod 2
\]
\[
pl_8(8n+5)\equiv pl_8(8n+6)\equiv pl_8(8n+7)\equiv 0\pmod 2
\]
\end{theorem}

Before proving Theorem \ref{mod4equiv}, we state the following lemma, which follows from Lemma \ref{equivlemma} due to the fact that $(1-q^j)^2 - (1-q^{2j}) \equiv (1-q^j)^2 + (1-q^{2j}) \pmod{2}$. 
\begin{lemma}\label{mod4equivlemma}
For any positive integer $j$,
\[
(1-q^j)^4\equiv (1-q^{2j})^2 \pmod {4}.
\] 
\end{lemma}
\begin{proof}[Proof of Theorem \ref{mod4equiv}]
To prove the first congruence of Theorem \ref{mod4equiv} we write the power series of $PL_4$ in terms of $PL_2$. To that end, note that by Lemma \ref{mod4equivlemma}
\begin{align*}
\sum_{n=0}^\infty pl_4(n)q^n &\equiv \frac{1}{(1-q)(1-q^2)^2(1-q^3)^3}\prod_{i=4}^\infty \frac{1}{(1-q^{2i})^2} \pmod 4 \\
& =  \frac{(1-q^4)^2(1-q^6)^2}{(1-q)(1-q^3)^3} \prod_{i=1}^\infty \frac{1}{(1-q^{2i})^2} \pmod 4\\
&\equiv  \frac{(1-q^2)^4(1-q^3)}{(1-q)} \prod_{i=1}^\infty \frac{1}{(1-q^{2i})^2} \pmod 4.
\end{align*}
Factoring $(1-q^2)$ and regrouping we have
\begin{align*}
\sum_{n=0}^\infty pl_4(n)q^n &\equiv (1+q)^3(1-q)^2(1-q^3)\left[ (1-q^2)\prod_{i=1}^\infty \frac{1}{(1-q^{2i})^2}\right] \pmod 4 \\
&\equiv  (1+q)^3(1-q)^2(1-q^3)  \left[ \sum_{n=0}^\infty pl_2(n) q^{2n}\right] \pmod 4.
\end{align*}
We explicitly compute $(1+q)^3(1-q)^2(1-q^3)$  to obtain
\[
\sum_{n=0}^\infty pl_4(n)q^n \equiv (1+q+2q^2+q^3+3q^5+2q^6+3q^7+3q^8)  \sum_{n=0}^\infty pl_2(n) q^{2n}  \pmod 4 .
\]
Now it is clear that
\[
pl_4(4n+1)\equiv pl_2(2n) + pl_2(2n-1)+3pl_2(2n-2)+3pl_2(2n-3) \pmod 4
\]
\[
pl_4(4n+2) \equiv pl_2(2n+1)+2pl_2(2n)+2pl_2(2n-2)+3pl_2(2n-3) \pmod 4 
\]
and
\[
pl_4(4n+3)\equiv pl_2(2n+1)+pl_2(2n)+3pl_2(2n-1)+3pl_2(2n-2) \pmod 4.
\]
Using the congruence \eqref{mod2case} we conclude that
 \begin{multline*}
pl_4(4n+1)-pl_4(4n+2)-pl_4(4n+3) \equiv \\
-2(pl_2(2n)+pl_2(2n+1))-2(pl_2(2n-1)+pl_2(2n-2)) \equiv 0 \pmod 4,
\end{multline*} 
which establishes the first congruence of Theorem \ref{mod4equiv}.

To prove the second congruence, we use Lemma \ref{equivlemma} to write
\begin{align*}
\sum_{n=0}^\infty pl_4(n)q^n &= \frac{1}{(1-q)(1-q^2)^2(1-q^3)^3}\prod_{i=4}^\infty \frac{1}{(1-q^i)^4} \\
&\equiv \frac{(1-q^4)(1-q^8)(1-q^{12})}{(1-q)(1-q^2)^2(1-q^3)^3}\prod_{i=1}^\infty \frac{1}{1-q^{4i}} \pmod 2 \\
&\equiv \frac{(1-q)^4(1-q^2)^4(1-q^3)^4}{(1-q)(1-q^2)^2(1-q^3)^3}\prod_{i=1}^\infty \frac{1}{1-q^{4i}} \pmod 2 \\
&\equiv (1-q)^3(1-q^2)^2(1-q^3) \prod_{i=1}^\infty \frac{1}{1-q^{4i}} \pmod 2.
\end{align*}
Since the product $\prod_{i=1}^\infty\frac{1}{1-q^{4i}}$ yields a series in $q$ with exponents which are multiples of $4$, to prove the congruence it suffices to explicitly compute the coefficients of $(1-q)^3(1-q^2)^2(1-q^3)$ from terms of the form $q^{4j+3}$. Indeed, the relevant terms are $4q^3+4q^7$, which are both congruent to $0$ modulo 2, establishing the congruence. \\
The final congruence is proved in a similar way
\begin{align*}
\sum_{n=0}^\infty pl_8(n) q^n &\equiv \frac{1}{(1-q)\cdots (1-q^7)^7} \prod_{i=8}^\infty \frac{1}{1-q^{8i}} \pmod 2 \\
& \equiv \frac{(1-q^8)(1-q^{16})\cdots (1-q^{56})}{(1-q)\cdots (1-q^7)^7} \prod_{i=1}^\infty \frac{1}{1-q^{8i}} \pmod 2 \\
&\equiv (1-q)^7(1-q^2)^6(1-q^3)^5\cdots (1-q^7) \prod_{i=1}^\infty \frac{1}{1-q^{8i}} \pmod 2 .
\end{align*}
 Again, explicit calculation of the coefficients of  $(1-q)^7(1-q^2)^6(1-q^3)^5\cdots (1-q^7)$ shows that the coefficients of $q^{8n+5}$, $q^{8n+6}$, and $q^{8n+7}$ are all even, which  yields the desired result.

\end{proof}

\subsection{Plane partition congruences related to multipartition congruences}\label{mod5section}

Here we prove the following two congruences of a different flavor, where the modulus and number of components are relatively prime.   
\begin{theorem}\label{mod5equiv}
The following hold for all $n\geq 0$.
\[
pl_2(5n+3) \equiv 0 \pmod 5 
\]
\[
pl_2 (5n+4) \equiv 0 \pmod 5.
\]
\end{theorem}

To prove this theorem we make use of a result of Kiming and Olsson \cite{KIM92} on multipartition congruences.  A {\it{$k$-component multipartition}} of a positive integer $n$ is defined to be a $k$-tuple of partitions $\lambda=(\lambda_1, \ldots, \lambda_k)$ such that $\sum_{i=1}^k|\lambda_i|=n$.  We note that some $\lambda_i$ may equal $\emptyset$.  

Let $p_k(n)$ count the number of $k$-component multipartitions of $n$, with the convention that $p_k(0)=1$.  Then $p_k(n)$ is generated by the $k$th power of the generating function for $p(n)$.  Namely,
\begin{equation*}
\sum_{n=0}^\infty p_k(n)q^n=\prod_{n=1}^\infty\frac{1}{(1-q^n)^k}.
\end{equation*}

\begin{theorem}[Kiming and Olsson \cite{KIM92}]\label{KimingOlson}
Let $\ell\geq 5$ be prime.  Then $p_{\ell - 3}(\ell n + a)\equiv 0 \pmod{\ell}$ for all $n\geq 0$ if and only if 
\[
\leg{8a+1}{\ell} \neq 1,
\]
where $\leg{\cdot}{\cdot}$ denotes the Legendre symbol.
\end{theorem}

We are now able to prove Theorem \ref{mod5equiv}.
\begin{proof}
As an immediate consequence of Theorem \ref{KimingOlson}, we have for $\ell=5$ that 
\begin{equation}\label{pcase}
p_2(5n+2) \equiv p_2(5n+3) \equiv p_2(5n+4) \equiv 0 \pmod{5}.
\end{equation}
Now for the plane partitions, we have
\begin{align*}
\sum_{n=0}^\infty pl_2(n)q^n &= \frac{1}{(1-q)}\prod_{i=2}^\infty \frac{1}{(1-q^i)^2} \\
&= (1-q) \sum_{n=0}^\infty p_2(n)q^n \\
&= p_2(0) + \sum_{n=1}^\infty (p_2(n) - p_2(n-1))q^n.\\
\end{align*}
Thus, we have by \eqref{pcase} that for all $n\geq 0$,
\begin{align*}
pl_2(5n+4) &= p_2(5n+4) - p_2(5n+3) \equiv 0 \pmod{5},\\
pl_2(5n+3) &= p_2(5n+3) - p_2(5n+2) \equiv 0 \pmod{5},
\end{align*}
as desired.
\end{proof}

\section{Conclusion}

We have  established a new method to confirm $k$-component plane partitions based on a bounded number of calculations. In particular Theorem \ref{congtheorem} gave rise to new congruences \eqref{mod3case1} and \eqref{mod7case}. It is clear that these results are limited by the computational capabilities available. As such, we  would hope that our results can be further simplified by finding more truncated and feasible forms of the generating functions, or more powerful tools to deal with what is the essential problem of classifying the coefficients of $F_\ell(q)$.  Generalizing Lemma \ref{equivlemma} as in Lemma \ref{mod4equivlemma} may also  provide a means to extend our proof to arbitrary prime powers. We proved several congruences involving prime powers but no general procedure is evident.  An interesting problem lies in whether any  additional congruences of the form \eqref{mod3case1} exist. Numerical calculations confirm that there exists  an  $n$ such that $pl_\ell(n\ell+\alpha)\not\equiv 0 \pmod{\ell}$ for all primes $\ell\leq 113$ and all $\alpha$. 

%\bibliographystyle{plain}
%\bibliography{mattbib}

\end{document}